\documentclass[12pt]{article}
\usepackage{hyperref}
\usepackage{cmap}                        
\usepackage[cp1251]{inputenc}            
\usepackage[english]{babel}
\usepackage[left=1in,right=1in,top=1in,bottom=1.5in]{geometry} 
\usepackage{amssymb,amsmath, amsthm, amscd,ifthen}
\usepackage{graphicx}
\usepackage{epigraph, xcolor, hyperref}

\newcommand{\R}{{\mathbb R}}

\newtheorem*{thm*}{Theorem}
\newcommand{\ff}{{\mathcal F}}
\newcommand{\aaa}{{\mathcal A}}

\newcommand{\G}{{\mathcal G}}
\newcommand{\V}{{\mathcal V}}

\newtheorem*{cla*}{Claim}
\newcommand{\bb}{{\mathcal B}}

\newtheorem{thm}{Theorem}

\newtheorem{opr}{Definition}
\newtheorem{lem}[thm]{Lemma}
\newtheorem{cla}[thm]{Claim}

\date{}

\title{Intersection theorems for $(-1,0,1)$-vectors}

\author{Peter Frankl\thanks{R\'enyi Institute, Budapest, Hungary and MIPT, Moscow. Research supported by the National Research, Development and Innovation Office -- NKFIH under the grant no.\ K 132696.}
\ and Andrey Kupavskii\thanks{MIPT, Moscow, IAS, Princeton and CNRS, Grenoble. Research supported by the grant of the Russian Government N 075-15-2019-1926.}}

\date{}

\begin{document}
\maketitle
\begin{abstract}
  In this paper, we investigate Erd\H os--Ko--Rado type theorems for families of vectors from $\{0,\pm 1\}^n$ with fixed numbers of $+1$'s and $-1$'s. Scalar product plays the role of intersection size. In particular, we sharpen our earlier result on the largest size of a family of such vectors that avoids the smallest possible scalar product. We also obtain an exact result for the largest size of a family with no negative scalar products. \end{abstract}
Intersection theorems form a classical part of discrete mathematics. They deal with families of sets in which pairwise intersection sizes of pairs of sets are restricted. One may forbid different intersection patterns. Forbidding all large intersections is the domain of coding theory. Forbidding all small intersections  lead to the famous theorems of Erd\H os--Ko--Rado and Ahswede--Khachatrian. Forbidding one intersection leads to Frankl--Wilson theorem. All these directions are very fruitful, and the corresponding results had great impact on combinatorics, discrete geometry and computer science. Let us dwell on some of these developments.

Set $[n]:=\{1,\ldots, n\}$ and let $2^{[n]},$ ${[n]\choose k}$ stand for the power set of $[n]$ and the set of all $k$-element subsets of $[n]$, respectively.  The Erd\H os--Ko--Rado theorem \cite{EKR} states that for $n\ge 2k$ the largest family $\ff\subset {[n]\choose k}$ in which any two sets intersect has size at most ${n-1\choose k-1}$.

Later, this theorem was extended to the case of {\it $t$-intersecting} families in ${[n]\choose k}$, i.e., the families in which any two sets intersect in $t$ elements. After a series of important developments \cite{F1}, \cite{W}, \cite{FF}, Ahlswede and Khachatrian \cite{AK} settled the conjecture due to the first author and determined the largest $t$-intersecting family in ${[n]\choose k}$ for any $n,k,t$. This theorem proved to be useful in, e.g., Hardness of Approximation \cite{DS}.

Wilson and the first author \cite{FW} obtained a surprising and powerful result using linear algebra, one particular case of which can be stated as follows: If $k$ is a prime power and no two sets in a family  $\ff\subset {[4k]\choose 2k}$ intersect in exactly $k$ elements, then $|\ff|\le 2{4k\choose k}$. That is, the size of any such $\ff$ is exponentially smaller than ${4k\choose 2k}$. This theorem has been influential in discrete geometry, where it implied exponential and sub-exponential lower bounds for the chromatic number of the space and Borsuk's problem, respectively, as well as in Ramsey theory, where it gave the best known explicit constructions of graphs avoiding large cliques and independent sets. We also note that this theorem is essentially sharp for families of sets.

In \cite{Rai1}, Raigorodskii applied the techniques of Frankl and Wilson to a more general collection of vectors and got improvements for both of the above discrete-geometric problems. Instead of working with sets, or vectors from $\{0,1\}^n$, he suggested to work with $\{0,\pm 1\}^n$. In this setting, however, the notion of intersection is ambiguous. The most appropriate notion for the applications in geometry was that of scalar product. At the same time, it turned out to be very challenging to extend extremal set theory techniques to this more geometric setting. In particular,  the lower bounds he got on the sizes of families in $\{0,\pm 1\}$ with one forbidden scalar product and the known upper bounds for the sizes of such families are exponentially far apart. For other related developments cf. \cite{K, PR}.

These developments and challenges motivated us to start a more systematic study of intersection theorems for $\{0,\pm 1\}$-vectors. There are different possible directions to pursue. In \cite{FK1, FK11}, we obtained  Erd\H os--Ko--Rado-type results for the families of $\{0,\pm 1\}$-vectors with fixed numbers of coordinates of each type. In \cite{FK5} (see also the arXiv version or  the corrigendum \cite{FK20}), we studied $t$-intersecting families of vectors with a fixed number of nonzero coordinates. Recently, Cherkashin and Kiselev
\cite{CK} obtained analogues of one forbidden intersection problem for the latter type of vectors.

\section{Fixed number of coordinates of each type}

\begin{opr}\label{def1} For $0\le \ell,k<n$ let $\mathcal V(n,k,\ell) \subset \R^n$  be the set of all $\{0,\pm 1\}$-vectors with exactly $k$ coordinates equal to $+1$ and $\ell$ coordinates equal to $-1$.
 \end{opr}

In what follows, we assume that $n\ge k+\ell$ and $k>\ell$. Thus, the minimum possible scalar product (denoted by $\langle .,.\rangle$) between two vectors from $\mathcal V(n,k,\ell)$ is $-2\ell$. Let us put $$\mathcal G(n,k,\ell) :=\big\{\mathcal V\subset \mathcal V(n,k,\ell): \langle \mathbf v,\mathbf w\rangle>-2\ell\text{ for any }\mathbf v,\mathbf w\in \V\big\},$$ $$ g(n,k,\ell):=\max_{\G\in \G(n,k,\ell)}|\G|.$$
We refer the reader to \cite{FK1} for a detailed discussion of the question of determining $g(n,k,\ell)$. Let us just mention the main results of \cite{FK1, FK11}. The first theorem completely determines $g(n,k,1)$.

 \begin{thm}[Frankl, Kupavskii \cite{FK1}]\label{fkjcta} We have
\begin{align}\label{eq001}g(n,k,1)=&  k{n-1\choose k} \text{\ \ \ \ \ \ \ for\ }2k\le n\le k^2,\\
\label{eq002}g(n+1,k,1)-g(n,k,1)=& {n\choose k}  \text{\ \ \ \ \ \ \ \ \ \ \ \ \ for\ } n> k^2.
\end{align}
\end{thm}
The second theorem provides some partial results for general $\ell$.
\begin{thm}[Frankl, Kupavskii \cite{FK11}]\label{fkbar} We have
\begin{equation}\label{eq000}{n\choose k+\ell}{k+\ell-1\choose \ell-1}\le g(n,k,\ell)\le {n\choose k+\ell}{k+\ell-1\choose \ell-1}+{n\choose 2\ell}{2\ell\choose \ell}{n-2\ell-1\choose k-\ell-1}.\end{equation}
If $2k\le n\le 3k-\ell$, then \begin{equation}\label{eq111} |\ff| ={n-1\choose k+\ell-1}{k+\ell-1\choose\ell}.\end{equation}
\end{thm}
The difficulty of the problem probably comes from the fact that there are two completely different extremal constructions, each being the best in a certain range. The first construction is EKR-type: simply take all vectors having $1$ in the first coordinate. It is extremal in the settings of \eqref{eq001} and \eqref{eq111}. The other construction is inductive: take the best construction for $n$ and append $0$ to all vectors. Then add all vectors that have $-1$ in the $(n+1)$'st coordinate. It is extremal, e.g., in the setting \eqref{eq002}.

Recently, our attention was drawn again to $g(n,k,\ell)$ because of a connection with the following problem:\footnote{See  \cite{FK5} for more details} for integers $n,s,x$, find the largest value of $|\ff|$, where $\ff$ is a family of $\{0,\pm 1\}$-vectors in $\R^n$ with exactly $s$ nonzero coordinates for an integer $x$ satisfies $\langle \mathbf v,\mathbf w\rangle \ge x$ for any $\mathbf v,\mathbf w\in \ff$. Interestingly, the answer to this question for odd negative $x$ depends on $g(n,s+\frac{x-1}2, \frac{1-x}2)$. In short, an answer to a $t$-intersecting-type question for one type of vectors depends on the answer to the $1$-intersecting-type question for another type of vectors.

One unsatisfying aspect of \eqref{eq000} is that the error term depends on $n$. One of the main results in this paper is the following theorem in the spirit of \eqref{eq002}.
\begin{thm}\label{thmmain1}
  We have \begin{equation}\label{eqplus} g(n+1,k,\ell)-g(n,k,\ell) = {n\choose k+\ell-1}{k+\ell-1\choose \ell-1}\end{equation}
  in each of the following cases:
  \begin{itemize}
    \item $n\ge 5k^2$ and $k>\ell+1$;
    \item $n\ge 2k^3$ and $k = \ell+1$.
  \end{itemize}
\end{thm}
Note that, in particular, this theorem allows to determine the value of $g(n,k,\ell)$ up to an additive function $F(k,\ell)$, independent of $n$. In what follows, we assume that $\ell\ge 2$ and thus $k\ge 3$, since the case $\ell =1$ is resolved by Theorem~\ref{fkjcta}. Hence $2k^3>5k^2$, so the second restriction on $n$ is stronger.
\vskip+0.1cm {\bf Remark. } The actual threshold on $n$ in Theorem~\ref{thmmain1} should probably be $n\ge \frac{(k+\ell-1)(k+\ell)}\ell$: this is the moment when the inductive extremal construction starts to give bigger increment than the EKR-type construction. Although the bounds that we got for $n$ are not so far off, finding the exact range seems to be out of reach for the present methods.

\subsection{Nonnegative scalar product for $\V(n,k,\ell)$}
Let $$m(n,k,\ell):=\max\big\{|\V|: V\subset \V(n,k,\ell), \langle \mathbf v,\mathbf w\rangle \ge 0\text{ for all }\mathbf v,\mathbf w\in \V\big\}.$$
{\bf Construction.} Let $[n] = X\sqcup Y$. Define
$$\V(X,Y):=\big\{v\in \V(n,k,\ell): v_i\ge 0\text{ for } i\in X\text{ and }v_i\le 0\text{ for } i\in Y\big\}.$$
If $\mathbf v,\mathbf w\in \V(X,Y)$ then $\langle \mathbf v,\mathbf w\rangle\ge 0$. Note that $|\V(X,Y)| = {|X|\choose k}{|Y|\choose \ell}$. This quantity is maximized when $|X|\sim \frac k{k+\ell} n$, $|Y|\sim \frac {\ell}{k+\ell}n$. Let us denote $$p(n,k,\ell):=\max_{X\subset [n]} \big|\V(X,[n]\setminus X)\big|.$$
Interestingly, while determining $g(n,k,\ell)$ even for large $n$ is a challenge, one can determine the value of $m(n,k,\ell)$ for large $n$ relatively easily.
\begin{thm}\label{thmmain2}
  $m(n,k,\ell) = p(n,k,\ell)$ holds for $n\ge 3n_0$, where $n_0 = (k+\ell)2^{k+\ell+2}$.
\end{thm}

\section{Proof of Theorem~\ref{thmmain1}}
\subsection{Shifting}
Given two vectors $\mathbf v=(v_1,\ldots,v_n),\mathbf w=(w_1,\ldots, w_n)\in \R^n$, we say that {\it $\mathbf v$ precedes $\mathbf w$ in the shifting partial order} and write $\mathbf v\prec \mathbf w$, if one can obtain $\mathbf v$ from $\mathbf w$ via a sequence of $(i\leftarrow j)$-shifts. For a vector $\mathbf w$ and $i<j$, the $(i\leftarrow j)$-shift is an operation that replaces $w_i,w_j$ with $\max\{w_i,w_j\}$ and $\min\{w_i,w_j\}$, respectively. Note that, whenever for two vectors $\mathbf v,\mathbf w$ we have $\mathbf w\prec \mathbf v$, then $\{v_1,\ldots,v_n\} = \{w_1,\ldots,w_n\}$ as multisets.

We call a family $\V\subset \R^n$ {\it shifted} if, whenever $\mathbf w\in \V$ and $\mathbf v\prec \mathbf w$, we have $\mathbf v\in\V$. \\



In \cite{FK1}, we have shown that the value of $g(n,k,\ell)$ is attained on a shifted family, and thus, in what follows, we restrict our attention to shifted subfamilies of $\V(n,k,\ell)$ only.\\

\subsection{Interlacedness}
Let us define the {\it degree of interlacedness} $\lambda(\mathbf w)$ of a $\{0,\pm 1\}$-vector $\mathbf w=(w_1,\ldots,w_n)$ as the minimum of $\sum_{j\ge i}w_j$, where the minimum is taken over all $i\le n$.
Let $S(\mathbf v)$ be the set of all $i$ such that $v_i\ne 0$. Let the {\it negative support $S_-(\mathbf v)$ of $\mathbf v$}   be the set of all $i$ such that $v_i=-1$. Define the {\it positive support} $S_+(\mathbf v)$ similarly.
\begin{lem}\label{lemfra}
  Suppose that $\mathbf w\in \V(n,k,\ell)$, $k\ge \ell$, and
  \begin{itemize}
    \item[(i)] $\lambda(\mathbf w)\ge 0$;
    \item[(ii)] $\big|S_+(w)\cap [2t-1]\big| \le t-1$ for all $t\ge 1$.
  \end{itemize}
  Then there exists $\mathbf v\in\V(n,k,\ell)$ such that $\mathbf v\prec \mathbf w$ in shifting order and $\langle \mathbf v,\mathbf w\rangle  = -2\ell$.
\end{lem}
\begin{proof}
  Let $q_1>q_2>\ldots>q_{\ell}$ be the indices with $w_{q_i} = -1$ and $p_1>p_2>\ldots>p_k$ be the indices with $w_{p_i} = 1$. Let $r_1<r_2<\ldots<r_k$ be the first $k$ indices not equal to $1$.
  \begin{cla}\label{clasim}
    $q_i<p_i$ for $i = 1,2,\ldots,\ell$ and $p_{k-i+1}>r_i$ for $i\in [k]$.
  \end{cla}
  \begin{proof}
    Indeed, if $q_i>p_i$ then $\sum_{j\ge q_i} w_j \le -1$,
    contradicting (i). Similarly, if $p_{k-i+1}<r_i$ then $\big|S_+(\mathbf w)\cap [p_{k-i+1}]\big|= i$, while $p_{k-i+1}\le 2i-1$, which contradicts (ii) for  $t=i$.
  \end{proof}
  Let $\mathbf u$ be the vector obtained from $\mathbf w$ by switching each of the $\ell$ positions $q_i$, $1\le i\le \ell$ with certain positions $p_{s(i)}$. Namely, for each $i=1,\ldots, \ell$ we perform a $(q_i\leftarrow p_{s(i)})$-shift, where $p_{s(i)}$ is the smallest such that, first, $p_{s(i)}>q_i$ and, second, $p_{s(i)}$ was not switched with some $q_{i'}$ for $i'<i$. (The other coordinates remain intact.) Thus, for $j=q_i$ or $p_{s(i)}$ we have $u_j = -w_j$. Note that $s\le i$ due to Claim~\ref{clasim} and thus such a shift is possible, and $\mathbf u\prec \mathbf w$.

  Let $J = \{j_1,\ldots, j_{k-\ell}\}$, $j_1<\ldots< j_{k-\ell}$, be the set of the first $k-\ell$ indices $j$ such that $w_j=0$. Note that this sequence is obtained from $r_1< \ldots<r_k$ by taking the first $k-\ell$ indices  that do not belong to $\{q_1,\ldots, q_{\ell}\}$. Let also $J' = \{j'_1,\ldots, j'_{k-\ell}\}$, $j_1<\ldots< j_{k-\ell}$, be the set of indices $j'$ such that $w_{j'} = u_{j'}=1$. These are the indices that were not touched when passing from $\mathbf w$ to $\mathbf u$.
  \begin{cla}\label{clasim2}
  $j_i<j'_i$ for any $i=1,\ldots, k-\ell$.
  \end{cla}
\begin{proof}
Arguing indirectly, assume that $j_i>j'_i$ for some $i$. Then, clearly, $|J'\cap [j'_i]|>\big|J\cap [j'_i]\big|$. By the definition of $\mathbf u$, for each $s\in [j'_i-1]$ such that $w_s = -1$, we found the smallest not yet used index $s'>s$, such that $w_{s'} = 1$, and switched them. Therefore, for any such coordinate $s$, we have $s'<j_i'$, since otherwise $j'_i$ should have been chosen as $s'$ and consequently could not belong to $J'$. In other words, this implies that there is a pairing of $+1$- and $-1$-coordinates in $[j'_i-1]\setminus J'$, which implies that the following holds:
\begin{equation}\label{zerosum}\sum_{s\in [j'_i]\setminus (J\cup J')}w_s = 0.\end{equation}
Each of $w_i\in [j'_i]\setminus (J\cup J')$ is either $+1$ or $-1$ and, using \eqref{zerosum}, the number of $+1$'s and $-1$'s  in $[j'_i]\setminus (J\cup J')$ is equal. Together with $|J'\cap [j'_i]|>\big|J\cap [j'_i]\big|$, it gives that more than a half of indices $i\in[j'_i]$  satisfy $w_i=1$, a contradiction with (ii) for $j'_i$ playing the role of $t$.
\end{proof}
Now we are in a position to finish the proof of the lemma.
Define $\mathbf v$ by applying $(j_i\leftarrow j'_i)$-shifts to $\mathbf u$ for all $i = 1,\ldots, k-\ell$. 
By the second claim, we get $\mathbf v\prec \mathbf u\prec \mathbf w.$ Moreover, it should be clear that $\langle \mathbf v,\mathbf w\rangle  = -2\ell$.
  \end{proof}

\subsection{Proof of Theorem~\ref{thmmain1}}
  It is easy to see the ``$\ge$'' part of the statement. Indeed, take the family $\V\in \G(n,k,\ell)$ of size $g(n,k,\ell)$. Extend it to a family $\V'\in\G(n+1,k,\ell)$ by appending $0$ as the $n+1$-th coordinate to each vector in $\V$ and adding all vectors from $\V(n,k,\ell)$ that have $-1$ on the $(n+1)$'st coordinate.\footnote{This is the inductive construction from the previous section.} It is easy to see that the number of the vectors in the latter group is ${n\choose k+\ell-1}{k+\ell-1\choose \ell-1}$.\\

  We go on to the ``$\le$'' part. Take any $\V\in \G(n+1,k,\ell)$. We may w.l.o.g. assume that $\V$ is shifted, and thus, by Lemma~\ref{lemfra}, one of the following two properties must hold for any $\mathbf v\in \V$:
  \begin{align}\label{eqkup} \text{either }&\big|S_+(\mathbf v)\cap [2t-1]\big| = t\text{ for some }t\ge 1\\\label{eqfra}\text{or }& \lambda(\mathbf v)\le -1. \end{align}
   Partition $\V$ into subcollections $\aaa:=\V(n+1^-)$, $\V(n+1^0)$, and $\bb:=\V(n+1^{+})$ of vectors from $\V$ that have $-1$, $0$, and $1$ on the $n+1$'st coordinate position, respectively. Clearly, $|\V(n+1^0)|\le g(n,k,\ell)$. Thus, to prove the theorem, it is sufficient to show that
  \begin{equation}\label{eqmainineq}
    |\aaa|+|\bb|\le {n\choose k+\ell-1}{k+\ell-1\choose \ell-1} = \big|\V(n+1,k,\ell)(n+1^-)\big| = \big|\V(n,k,\ell-1)\big|.
  \end{equation}
Put
$$ i'(\mathbf v) := \max\Big\{i: \sum_{j\ge i} v_j = -1\Big\}, \ \ \ i(\mathbf v):=\big|S_-(\mathbf v)\cap [i'(\mathbf v),n+1]\big|.$$
Note that $i'(\mathbf v)$ is defined for any vector in $\bb$ not satisfying \eqref{eqkup} due to the fact that then \eqref{eqfra} must hold. If this is the case, note also that $2\le i(\mathbf v)\le \ell$ for any $\mathbf v \in \bb$, since $v_{n+1} = 1$. We have $|S_+(\mathbf v)\cap [i'(\mathbf w),n+1]| = i(\mathbf v)-1$ by the definition of $i'(\mathbf v)$. 
We cover $\bb$ by the following families $\bb_1^t,\bb_2^{j,j'}$:
$$\bb_1^{t,m}:=\Big\{\mathbf v \in \bb: 
\mathbf v \text{ satisfies }\eqref{eqkup}\text{ with } t, \text{ and } \big|S_-(\mathbf v)\cap [t]\big| = m \Big\},$$
$$\bb_2^{j,j'}:=\big\{\mathbf v\in \bb\setminus \bb_1: i(\mathbf v)=j,i'(\mathbf v)=j'\big\}.$$
Note that $\bb_2^{j,j'}$ are disjoint for different $j,j'$.

The idea behind such covering is quite simple: for each of these parts, we can find a subfamily of $\V(n+1,k,\ell)(n+1^-)$ and a suitable bipartite graph in order to obtain a certain sum-type inequality. The families $\bb_1^t$ and $\bb_2^{j,j'}$ together are small, and we shall conclude that for the optimal choice of $\V$, one has to take $\aaa = \V(n+1,k,\ell)(n+1^-)$.\\

We first obtain a sum-type inequality for each of $\bb_1^{t,m}$. 
Consider the following two subfamilies of $\V(n+1,k,\ell)$:

\begin{small}\begin{align*}\mathcal Y^{t,m}:=&\Big\{\mathbf v\in \V(n+1,k,\ell)(n+1^+): |S_-(\mathbf v)\cap [2t-1]|=m, |S_+(\mathbf v)\cap [2t-1]| = t\}\Big\},\\
\mathcal X^{t,m}:=&\Big\{\mathbf v\in \V(n+1,k,\ell)(n+1^-): |S_+(\mathbf v)\cap [2t-1]|=m, |S_-(\mathbf v)\cap [2t-1]| = \max\{t-(k-\ell),0\}\Big\}.\end{align*}\end{small}
That is, $\bb_1^{t,m} = \bb\cap \mathcal Y^{t,m}$. Put
$\aaa^{t,m}:=\mathcal X^{t,m}\cap \aaa$.
Note that $\mathcal X^{t,m_1}$ and $\mathcal X^{t,m_2}$, and thus $\aaa^{t,m_1}$ and $\aaa^{t,m_2}$, are disjoint if $m_1\ne m_2$. Define a bipartite graph $G^{t,m}$ between $\mathcal X^{t,m}$ and $\mathcal Y^{t,m}$ that connects the vectors that have scalar product $-2\ell$. It is not difficult to check that this graph is non-empty (i.e., that there are  pairs $\mathbf v\in \mathcal X^{t,m}$, $\mathbf w\in\mathcal Y^{t,m}$ with $\langle \mathbf v,\mathbf w\rangle = -2\ell$). Moreover, the graph is clearly biregular due to the symmetry in the definition.

We shall use the following lemma, which can be proved by a simple averaging argument.
\begin{lem}\label{lembipaver}
  If $G = (A\cup B,E)$ is a bipartite biregular graph, then any independent set $I\subset A\cup B$ satisfies $|I\cap B|+\alpha |I\cap A|\le \alpha |A|$ for any $\alpha\ge\frac{|B|}{|A|}$.
\end{lem}

$\aaa^{t,m}\cup \bb^{t,m}_1$ is an independent set in $G^{t,m}$. Moreover,
$$\frac{|\mathcal Y^{t,m}|}{|\mathcal X^{t,m}|} = \frac{{2t-1\choose m}{2t-1-m\choose t}{n-2t+1\choose \ell-m}{n-2t+1-\ell+m\choose k-1-t}}{{2t-1\choose m}{2t-1-m\choose \max\{t-k+\ell,0\}}{n-2t+1\choose k-m}{n-2t+1-k+m\choose \ell-1-\max\{t-k+\ell,0\}}}.$$
Let us consider two cases. First, assume that $\max\{t-k+\ell,0\} = 0$, i.e., $t\le k-\ell$. 
Then we have
$$\frac{|\mathcal Y^{t,m}|}{|\mathcal X^{t,m}|} = \frac{{2t-1-m\choose t}{n-2t+1\choose \ell-m}{n-2t+1-\ell+m\choose k-1-t}}{{n-2t+1\choose k-m}{n-2t+1-k+m\choose \ell-1}} =
{2t-1-m\choose t}\cdot\frac{(k-m)!}{(k-1-t)!}\cdot \frac{(\ell-1)!}{(\ell-m)!} \cdot$$$$\frac{(n-2t+1-k-\ell+1+m)!}{(n-2t+1-k-\ell+1+m+t)!} <(2t-1)^{t}k^{t+1-m}\ell^{m-1}\cdot $$$$\frac 1{(n-2t+2-k-\ell)^{t}}\le \frac k\ell k^{t-m}\ell^m\Big(\frac{2t-1}{4k^2}\Big)^t\le \frac k{\ell}\Big(\frac{2t-1}{4k}\Big)^{t},$$
where the second to last inequality uses that $n\ge 5k^2$ and $k\ge 3$, implying $n-2t+2-k-\ell\ge n-3k\ge 4k^2$, and the last inequality uses $k^{t-m}\ell^m\le k^t$.

Thus, using Lemma~\ref{lembipaver}, we get that for any $t\le k-\ell$ we have
$$|\bb^{t,m}_1|+\frac k{\ell}\Big(\frac{2t-1}{4k}\Big)^{t}|\aaa^{t,m}|\le \frac k{\ell}\Big(\frac{2t-1}{4k}\Big)^{t}|\mathcal X^{t,m}|,$$
and, using that $\mathcal X^{t,m_1}$ and $\mathcal X^{t,m_2}$ for distinct $m_1,m_2$, we infer that
\begin{equation}\label{eqsum10}\Big|\bigcup_{m}\bb^{t,m}\Big|+\frac k{\ell}\Big(\frac{2t-1}{4k}\Big)^{t}|\aaa|\le \frac k{\ell}\Big(\frac{2t-1}{4k}\Big)^{t}\big|\V(n+1,k,\ell)(n+1^-)\big|\end{equation}

Assume next that $t\ge k-\ell+1$ and $m\ge 1$, and thus $\min\{t-k+\ell,0\} = t-k+\ell$.  (Recall that $t\le k$ should always hold.) Let us do an auxiliary estimate:
$$\frac{(t-k+\ell)!(t-1-m+k-\ell)!}{t!(t-1-m)!}= \prod_{i=1}^{k-\ell}\frac{t-1-m+i}{t-k+\ell+i}\le \Big(\frac{k-\ell+1}2\Big)^{k-\ell}.$$
Doing similar calculations as in the previous case, we obtain

$$\frac{|\mathcal Y^{t,m}|}{|\mathcal X^{t,m}|} = \frac{{2t-1-m\choose t}{n-2t+1\choose \ell-m}{n-2t+1-\ell+m\choose k-1-t}}{{2t-1-m\choose t-k+\ell}{n-2t+1\choose k-m}{n-2t+1-k+m\choose k-1-t}} = \frac{(k-m)!(t-k+\ell)!(t-1-m+k-\ell)!}{(\ell-m)!t!(t-1-m)!}\cdot$$
$$\frac{(n-2t+1-2k+1+m+t)!}{(n-2t+1-k-\ell+1+m+t)!}< \frac{k^{k-\ell}\Big(\frac{k-\ell+1}2\Big)^{k-\ell}}{(n-3k)^{k-\ell}}:=\alpha(n,k,\ell).$$

As before, we conclude that
\begin{equation}\label{eqsum11}\Big|\bigcup_{m}\bb^{t,m}_1\Big|+\alpha(n,k,\ell)|\aaa|\le \alpha(n,k,\ell)\big|\V(n+1,k,\ell)(n+1^-)\big|.\end{equation}
Finally, we note that if $t\ge k-\ell+1$ and $m=0$ then $\bb^{t,m}_1\subset \cup_{j,j'}\bb_2^{j,j'}$. Denote $$\mathcal \bb_1':= \bigcup_{m,t\ :\ t\le k-\ell}\bb_1^{t,m}\cup\bigcup_{t\ge k-\ell+1}\bb_1^{t,0}.$$

Next, we obtain a sum-type inequality for  $\bb_2^{j,j'}\setminus \bb_1'$. Note that, for any $\mathbf v\in \bb_2^{j,j'}$, we have $|S_+(\mathbf v)\cap [j'-1]| = k-j+1.$ 
For any  $\mathbf v\in \bb\setminus \bb_1'$ either the property \eqref{eqkup} is not satisfied for $t = j'-1$ or \eqref{eqkup} is satisfied for $t\ge k-\ell+1$ and, additionally, $|S_-(\mathbf v)\cap [j'-1] = 0|$. In the first case, we have $j'-1\ge 2(k-j+1)$ directly. In the second case, we have $j'-1\ge 2(k-\ell+1)-1 = 2(k-j+1)-1$. Thus, in any case, we may assume that
\begin{equation}\label{eqcondj'}
  j'-1\ge 2(k-j+1).
\end{equation}
Remark that either $\bb_2^{j,j'}\subset \bb_1'$ or $\bb_2^{j,j'}\cap \bb_1' = \emptyset$. In what follows, we only work with the latter case. We partition $\bb_2^{j,j'}$ into subfamilies as follows. 
For a vector $\mathbf v\in \bb_2^{j,j'}$, put $\mathbf w(\mathbf v):=\mathbf v|_{[j',n+1]}$. 
For any vector $\mathbf x$, let $\bar {\mathbf x}$ be the vector defined by $\bar x_i = -x_i$ for each coordinate. Let $\V(n+1,k,\ell,\mathbf w)$ be the subfamily of vectors $\mathbf v\in \V(n+1,k,\ell)$ with $\mathbf w(\mathbf v)=\mathbf w$, and put $\bb_2^{j,j'}(\mathbf w) = \V(n+1,k,\ell,\mathbf w)\cap \bb_2^{j,j'}$. Also put 
$\aaa(\bar{\mathbf w}):=\V(n+1,k,\ell,\bar{\mathbf w})\cap \aaa$. We note that, for different $\mathbf w$ (as well as for different $j,j'$), the families $\aaa(\bar{\mathbf w})$ are disjoint.

In what follows, we think of both $\bb_2^{j,j'}(\mathbf w)$ and $\aaa(\bar{\mathbf w})$ as of families of vectors in $\{0,\pm 1\}^{[j'-1]}$.
We have
\begin{align*}\bb_2^{j,j'}(\mathbf w)&\subset \V(j'-1, k-j+1, \ell-j)=:\mathcal U_1,\\ \aaa(\bar{\mathbf w})&\subset \V(j'-1, k-j, \ell-j+1):=\mathcal U_2.\end{align*}
Define a bipartite graph $G'$ between $\mathcal U_1$ and $\mathcal U_2$ by joining by an edge any two vectors $\mathbf u, \mathbf v$ that have scalar product $-2\ell+2j-1$ (the vectors from $\{0,\pm 1\}^{n+1}$ that ``shorten'' to $\mathbf u, \mathbf v$ then have scalar product $-2\ell$).  The set $\bb_2^{j,j'}(\mathbf w)\cup \aaa(\bar{\mathbf w})$ is independent in $G'$. Moreover, $G'$ is biregular and of non-zero degree due to \eqref{eqcondj'}. 
Next,
$$\frac{|\mathcal U_2|}{|\mathcal U_1|} = \frac{{j'-1\choose k+\ell-2j+1}{k+\ell-2j+1\choose \ell-j+1}}{{j'-1\choose k+\ell-2j+1}{k+\ell-2j+1\choose \ell-j}}= \frac{k-j+1}{\ell-j+1}\ge \frac {k}{\ell}.$$
This holds for any $\mathbf w$ and $j,j'>0$. Thus, using Lemma~\ref{lembipaver} and summing over all $j,j'$ such that $\bb_2^{j,j'}\cap \bb_1' = \emptyset$ and vectors $\mathbf w$, we get that
\begin{equation}\label{eqsum2} \sum_{j,j'>0}\big|\bb_2^{j,j'}\setminus \bb_1'\big|+\frac{\ell}{k}|\aaa|\le \frac{\ell}{k}\big|\V(n+1,k,\ell)(n+1)^+\big|.\end{equation}

All we are left is to sum up equations \eqref{eqsum10} for $1\le t\le k-\ell$, \eqref{eqsum11} for $k-\ell+1\le t\le k$, and \eqref{eqsum2}. We get that
\begin{small}\begin{equation*}\label{eqsum3}
|\bb|+\Big(\sum_{t=1}^{k-\ell} \frac k{\ell}\Big(\frac{2t-1}{4k}\Big)^{t} + \ell\alpha(n,k,\ell)+ \frac{\ell}k\Big)|\aaa|\le \Big(\sum_{t=1}^{k-\ell} \frac k{\ell}\Big(\frac{2t-1}{4k}\Big)^{t} + \ell\alpha(n,k,\ell)+ \frac{\ell}k\Big)|\V(n,k,\ell-1)|,
\end{equation*}\end{small}
and we are done as long as $\sum_{t=1}^{k-\ell} \frac k{\ell}\big(\frac{2t-1}{4k}\big)^{t} + \ell\alpha(n,k,\ell)+ \frac{\ell}k\le 1$. We note that the first sum is at most $2\frac k{\ell} \frac 1{4k} = \frac 1{2\ell}$ since $\frac{2t-1}{4k}<1/2$. For $k-\ell\ge 2$, using $n\ge 5k^2$, we get that $\ell\alpha(n,k,\ell)\le \ell\frac{k-\ell+1}{8k}\Big)^{k-\ell}\le \frac 1{7k}$. Thus, the coefficient in question is at most $\frac{1}{2\ell}+\frac {7\ell+1}{7k}<1,$ where the last inequality is valid for any $1\le \ell \le k-2$, $k\ge 3$. For $k-\ell = 1$, using $n\ge 2k^3$, we have $\ell\alpha(n,k,\ell)<\frac 1{2k}$, moreover, the first term in the coefficient in question is at most $\frac 1{4\ell} = \frac 1{4k-4}$. Summing up, the coefficient in question is $\frac 1{4k-4}+\frac 1{2k}+\frac{k-1}k = \frac 1{4k-4}+1-\frac 1{2k}<1$. 

\section{Proof of Theorem~\ref{thmmain2}}
  Let us prove the following statement by induction on $n$: for $n\ge n_0$, \begin{equation}\label{sick} m(n,k,\ell)\le \max\Big\{p(n,k,\ell), {n_0 \choose k}{n_0\choose \ell}+p(n,k,\ell)/2\Big\}.\end{equation}

  It clearly holds for $n = n_0$ since the cardinality of $\V(n_0,k,\ell)$ is at most the second term in the maximum.

  Now let $n>n_0$ and $\V\in V(n,k,\ell)$ be a collection without negative scalar products. If $\V\subset \V(X,Y)$ for some $X\sqcup Y = [n]$ then we are done. Thus, w.l.o.g., there are $\mathbf v, \mathbf w\in \V$ with $v_n = 1$, $w_n = -1$. By induction, the family $\V(n^0)$ satisfies \eqref{sick} with $n$ replaced by $n-1$. Thus, to verify the induction step, we have to show that
  $$|\V(n^+)|+|\V(n^-)|\le \frac 12\big(p(n,k,\ell)-p(n-1,k,\ell)\big).$$
  It should be clear from the definition that $p(n,k,\ell)-p(n-1,k,\ell)= \max \big\{p(n-1,k,\ell-1),p(n-1,k-1,\ell)\big\}\ge \frac 12\big(p(n-1,k,\ell-1)+p(n-1,k-1,\ell)\big)$.
  Thus, to prove the displayed inequality, it is sufficient for us to show that \begin{equation}\label{eqshow}|\V(n^+)|\le \frac 14 p(n-1,k-1,\ell),\end{equation} as well as $|\V(n^-)|\le \frac 14 p(n-1,k,\ell-1)$. We show only the former, because the latter can be proved in the same way.

  Any $\mathbf v'\in \V(n^+)$ satisfies $S(\mathbf v')\cap S(\mathbf w)\ne \emptyset$, and thus $|\V(n^+)|\le |\V(n-1,k-1,\ell)|-|\V(n-k-\ell,k-1,\ell)| \le {k+\ell-1\choose k}(k+\ell-1)
  {n-1\choose k+\ell-2}\le \frac {k+\ell}n|\V(n-1,k-1,\ell)|$. On the other hand, $p(n-1,k-1,\ell)\ge {(n-1)/2\choose k-1}{(n-1)/2\choose \ell}\ge
  2^{-k-\ell}|\V(n-1,k-1,\ell)|$. 
  Thus, \eqref{eqshow} holds provided $n\ge (k+\ell)2^{k+\ell+2} = n_0$.

  The only thing we are left to show is that for $n\ge 3n_0$ the maximum on the right hand side of \eqref{sick}  is attained by the first term. In other words, that $p(n,k,\ell)\ge 2{n_0\choose k}{n_0\choose \ell}$. However, $p(n,k,\ell)\ge p(3n_0,k,\ell)\ge
   {3n_0/2\choose k}{3n_0/2\choose \ell}>2{n_0\choose k}{n_0\choose \ell}$. This concludes the proof.

\end{document}